\documentclass[12pt]{amsart}
\usepackage{amssymb}
\usepackage{amsfonts}
\usepackage{amsmath}
\usepackage{amsthm}
\usepackage{remark}
\usepackage{enumerate}
\usepackage{mathrsfs} 
\usepackage{hyperref}
 
\newtheorem{thm}{Theorem}[section]

\newtheorem{lemma}[thm]{Lemma}
\newtheorem{corollary}[thm]{Corollary}
\newtheorem{proposition}[thm]{Proposition}
\newremark{definition}[thm]{Definition}
\newremark{remark}[thm]{Remark}
\newremark{example}[thm]{Example}
\newremark{notation}[thm]{Notation}
\newremark{algo}[thm]{Algorithm}
\newremark{test}[thm]{Test}

\newcommand{\disc}{\mathrm{disc}}
\newcommand{\ord}{\operatorname{ord}}
\newcommand{\Spec}{\operatorname{Spec}}
\newcommand{\Z}{\mathbb Z}
\newcommand{\p}{\mathfrak p}

\begin{document}

\title[Computing minimal Weierstrass equations]{\bf Computing minimal Weierstrass equations of hyperelliptic curves}

\author{Qing Liu}

\begin{abstract}
We describe an algorithm for determining a minimal Weierstrass equation
for hyperelliptic curves over principal ideal domains.  When the curve
has a rational Weierstrass point $w_0$, we also give a similar algorithm for
determining the minimal Weierstrass equation with respect to $w_0$. 
\end{abstract}

\thanks{I would like to thank Bill Allombert for clarifications regarding some computational aspects in this article and for pointing out related references. I would also like to thank the referees for their thorough reading. Thank you also to the referees and Bill Allombert for suggestions which led to improvements in the presentation of this manuscript.}


\address{Universit\'e de Bordeaux, Institut de Math\'ematiques de 
  Bordeaux, CNRS UMR 5251, 33405 Talence, France}

\address{\href{http://orcid.org/0000-0001-6884-139X}{Orcid iD 0000-0001-6884-139X}}

\email{Qing.Liu@math.u-bordeaux.fr}

\maketitle

Tate's algorithm \cite{Tate} determines the reduction type of elliptic
curves over discrete valuation rings with perfect residue field. In
particular it determines the minimal Weierstrass equation. Over number
fields with trivial class number, Laska \cite{Laska} gave a faster method
to determine the minimal Weierstrass equation.  

More generally, for any hyperelliptic curve $C$ of genus $g\ge 1$ over
a discrete valuation ring, there is a natural notion of minimal 
Weierstrass equations. This question is studied in \cite{Loc} and
\cite{LTR}. An algorithm for determining a
minimal Weierstrass equation is sketched in \cite{LTR}.
See also \cite{CFS}, \S 4, for hyperelliptic curves of
genus $1$ (not necessarily elliptic). Here we deal with hyperelliptic
curves of any genus $g\ge 1$ and also with those having a rational Weierstrass point (in
which case there is a notion of minimal \emph{pointed} Weierstrass
equation as for elliptic curves, see \cite{Loc} or below) over
principal ideal domains. 
The aim of the present work is to make our algorithm completely
explicit. It is now implemented over $\mathbb Z$ in PARI release 2.15
(\cite{pari}) by B. Allombert. 

Let us briefly present the content of this work. 
Let $A$ be a principal ideal domain with perfect residue fields at its
maximal ideals. Let $C$ be a hyperelliptic curve of genus $g\ge 1$ 
over $K:=\mathrm{Frac}(A)$, given by an integral Weierstrass equation
\begin{equation}
  \label{eq:start}
  y^2+Q(x)y=P(x) 
\end{equation} 
with $Q(x), P(x)\in A[x]$ such that $\deg Q\le g+1$ and $\deg P\le 2g+2$.
Such an equation is said to be \emph{minimal}, when $A$ is local, if its
discriminant $\Delta$ (\cite{LTR}, \S 2) has the smallest valuation among all
integral Weierstrass equations describing the same curve over $K$. 
When $A$ is global, the equation is said to be \emph{minimal} if it
is minimal at all localizations of $A$ at maximal ideals. It is well
known that such an equation exists because $A$ is principal, see
for instance \cite{LTR}, \S3, Proposition 2. Note that for a given $C$
there are finitely many minimal Weierstrass equations (up to the
natural action of $\mathrm{GL}_2(A)$), but in general they are not unique. 
For instance, an elliptic curve of type $I_n$ ($n\ge 1$) over a
discrete valuation ring has $n$ non-equivalent (Definition~\ref{sb-W})
minimal Weierstrass equations as (non pointed) hyperelliptic curve.

When $C$ has a rational Weierstrass point $w_0$, an integral
\emph{pointed Weierstrass} equation of $(C, w_0)$ is an
equation
$$
y^2+Q(x)y=P(x)
$$
over $A$ with $\deg Q\le g$, $P(x)$ monic of degree $2g+1$ and such
that $w_0$ is the pole of $x$. 
A minimal pointed Weierstrass equation exists and is unique up to the
transformations 
    $$x=u^2x_1 +c, \quad y=u^{2g+1}y_1+H(x)$$
    with $u\in A^*$, $c\in A$ and $H(x)\in A[x]$. See Lemma~\ref{p-c}.

The principle of the minimization ({\it i.e.}\ finding a minimal Weierstrass
equation) we use is to successively minimize at a finite list of bad
primes. At each bad prime $\p$, we first normalize the
equation ({\it i.e.}\ find a Weierstrass equation defining a normal
scheme). This is very simple at primes of odd residue characteristic,
but requires an appropriate algorithm otherwise 
(Algorithm~\ref{algo-normalization}). Then the minimality at $\p$ is checked by
computing the multiplicities $\lambda(p_0)$ at very special rational
points $p_0$ of the reduction mod $\p$ (see
Proposition~\ref{min-cond}). These points correspond to
roots of high order of some polynomials over the residue field at $\p$
(Lemma~\ref{l-d}). If the minimality condition is not satisfied, a new
candidate (normal) Weierstrass equation is given (\S~\ref{dilat}) and we can
restart the minimality checking.

To get a global minimal Weierstrass equation,  we notice that our
minimization process at a prime dividing $2$ does not change the
valuation of the discriminant 
of the initial equation at other primes. Similarly the minimization process at odd
primes will keep unchanged the discriminant of the initial equation at the other
odd primes, but it will affect the discriminant at primes dividing $2$. Nevertheless
we use an easy trick to combine minimal Weierstrass equations at odd
primes and at primes dividing $2$ (Lemma~\ref{odd-2-even}). 
The strategy is similar for pointed Weierstrass equations. The 
procedure is then much simpler because we only use transformations
of the form $x=u^2x_1+c$, $y=u^{2g+1}y_1+H(x)$. 
\medskip 

Note that it is essential here to suppose $A$ is principal (or at
least that the primes dividing the discriminant of an initial equation
over $A$ are principal), as otherwise a global minimal Weierstrass
equation may not exist. 
\medskip 

In \S~\ref{normalization} and \S~\ref{lambda} we explain the process
of normalization and the computation of the multiplicity $\lambda$.
In \S~\ref{min-crit} and \S~\ref{pointed-min} we give the minimality criterion
respectively for Weierstrass equations and pointed Weierstrass
equations. Finally algorithms to find minimal (resp. pointed)
Weierstrass equations are described in the last two sections.
\medskip 

\noindent {\bf Notation}   We denote by $A$ a principal ideal domain with 
field of fractions $K$ such that its residue fields at maximal ideals
are perfect. Primes of $A$ will be denoted by $\p$.
When $\p$ is fixed, we denote by $k=k(\p)$ the residue field at $\p$,
$p=\mathrm{char}(k)$ 
and $\pi\in A$ a generator of $\p$.

For any $H(x)\in A[x]$, its image in $k[x]$ is denoted by
$\bar{H}(x)$.

The normalized  valuation on $K$ defined by $\p$ will be denoted by
$v_\p$ or just $v$ if there is no ambiguity. 

An element $a\in A$ is \emph{odd} if $a\ne 0$ and $aA+2A=A$.
We call $\p$ an \emph{odd prime} if
$\p +2A=A$. Otherwise it is called an \emph{even prime}.

\begin{section}{Weierstrass models} 

\begin{definition}\label{sb-W}   
The projective scheme over $A$ defined by Equation~\eqref{eq:start} is
denoted by $W$. It is obtained by glueing the affine schemes 
$$
W_0=\operatorname{Spec} A[x,y]/(y^2+Q(x)y-P(x) )
$$
and
$$
W_\infty=\operatorname{Spec} A[t, z]/(z^2+t^{g+1}Q(t^{-1})z-t^{2g+2}P(t^{-1}))
$$
along the identification $t=1/x, z=y/x^{g+1}$. This scheme is integral
and flat over $A$, with generic fiber isomorphic to $C$. This is a
\emph{Weierstrass model of $C$ over $A$}. It is said to
    be \emph{normal} if $W$ is a normal scheme.  An \emph{isomorphism} of 
    Weierstrass models of $C$ is an isomorphism of 
  $A$-schemes compatible with the isomorphisms with $C$.

If $W\times_{\Spec A} \Spec A_\p$ satisfies some property, then we say
that $W$ satisfies this property \emph{at $\p$}.

We say that two Weierstrass equations of $C$ over $A$ are
\emph{equivalent at $\p$} if the associated Weierstrass models
are isomorphic over $A_\p$.   This implies that their discriminants
have the same valuation at $\p$ (see below). 
\end{definition} 

\begin{definition} 
Let $F(x)=4P(x)+Q(x)^2$ with leading coefficient $a$. 
Recall (\cite{LTR}, \S 2) that when $\mathrm{char}(K)\ne 2$,
the \emph{discriminant} $\Delta$ of Equation~\eqref{eq:start} is given by 
$$\Delta=\left\lbrace{
\begin{matrix} 
2^{-4(g+1)}\disc(F)\hfill & \quad \text{\rm if} \ \  \deg F=2g+2 \\ 
\noalign{\bigskip}
2^{-4(g+1)}a^2\disc(F) & \quad \text{\rm if} \ \  \deg F=2g+1.  
\end{matrix}
}\right.
$$
So if $Q=0$, then $\Delta=2^{4g}\disc(P)$.
\end{definition}
\medskip

Other integral Weierstrass equations of $C$ are obtained by change of variables
\begin{equation}
  \label{eq:cvar}
  x=\dfrac{ax_1+b}{cx_1+d}, \quad
y=\dfrac{ey_1+H(x_1)}{(cx_1+d)^{g+1}},
\end{equation} 
with $a, b, c, d, e\in A$, $H(x_1)\in A[x_1]$, and $ad-bc, e\ne 0$.  
The corresponding Weierstrass models are isomorphic if and only if
$ad-bc, e \in A^*$ (invertible). 

The discriminant $\Delta_1$ of the equation with $x_1, y_1$ is
given by
\begin{equation} \label{eq:d1} 
\Delta_1=e^{-4(2g+1)}(ad-bc)^{2(g+1)(2g+1)}\Delta. 
\end{equation} 
\end{section}

\begin{section}{Normalization}\label{normalization} 

  \begin{notation}  Let $\p$ be a prime of $A$.
For any $H(x)=\sum_i a_ix^i\in A[x]$, we denote by 
\begin{equation}
  \label{eq:v}
  v_\p(H):=\min_i \{ v_\p(a_i) \}\in \mathbb N\cup \{ +\infty\} 
\end{equation}
(or just $v(H)$). This is the Gauss valuation on $K(x)$ with respect
to the variable $x$ extending $v_\p$. 
\end{notation}

  Let $W$ be the Weierstrass model over $A$ defined by Equation~\eqref{eq:start}.
  
\begin{lemma}[Normalization away from $2$]\label{normalize-n2}
Suppose that $\mathrm{char}(K)\ne 2$.  Let $e^2$ be a biggest odd square
factor of the content $\operatorname{cont}(4P+Q^2)$ of $4P+Q^2$.   Then 
the equation 
$$z^2=e^{-2}(4P(x)+Q^2(x))$$
defines the normalization of $W$ at all odd primes of $A$.
 \end{lemma} 

\begin{proof} See \cite{LTR}, Lemme 2(d), p. 4582.
\end{proof}

\begin{lemma}[Normalization at even primes]\label{normalize_at_2} 
  Let $\p$ be an even prime of $A$.
  \begin{enumerate}[{\rm (1)}] 
  \item   If one of the following conditions is satisfied, then $W$ is
    normal at $\p$: 
  \begin{enumerate}[{\rm (a)}]
  \item $v(Q)=0$ (then $W_k$ is reduced); 
  \item $v(Q)>0$ and $\bar{P}(x)$ is not a square in $k[x]$;
  \item $v(Q)>0$ and $v(P)=1$. 
  \end{enumerate}
\item  If the pair $Q(x), P(x)$ satisfies one of the above conditions, then 
  the pair $x^{g+1}Q(1/x), x^{2g+2}P(1/x)$ satisfies the same condition. 
\item If the pair $Q(x), P(x)$ satisfies none of the conditions of {\rm (1)},  there exists a change of variables
  $$y=\pi^n y_1 + H(x)$$
 with $n\ge 0$ and $H(x)\in A[x]$ such that 
 $Q_1(x):=\pi^{-n}(Q-2H)$, $P_1(x):=\pi^{-2n}(P+QH-H^2)$ belong to
 $A[x]$ and satisfy one of the conditions of {\rm (1)}. The 
  new equation
  $$ y_1^2+Q_1(x)y_1=P_1(x)$$
with $y_1=\pi^{-n}(y+H)$, then defines the normalization of $W$.  
\end{enumerate}
\end{lemma} 

\begin{proof} (1) The normality under (a) or (b) holds by  
 \cite{LTR}, Lemme 2(a)-(b), p. 4582. Use the same lemma, Part (c)
 when Condition (c) is satisfied. 

 (2) is straightforward.
 
 (3) The construction of $y_1$ is given in
 Algorithm~\ref{algo-normalization}. 
\end{proof}

\begin{remark}  
  \begin{enumerate}[{\rm (1)}]
  \item The transformation in Lemma~\ref{normalize_at_2}(3) does not
    affect the discriminant at other primes than $\p$, as the 
    new discriminant is the former discriminant divided
    by a power of $\pi$. On the other
hands, the normalization process at odd primes in Lemma~\ref{normalize-n2}
multiplies the discriminant by $(2e^{-1})^{4(2g+1)}$, therefore it 
does modify its even part. 
\item The converse of Lemma~\ref{normalize_at_2}(1) is
false. For example the equation $y^2=x^6+1$ over $\Z_2$ is normal (even minimal), 
but it does not satisfy any of the conditions of
\ref{normalize_at_2}(1). However the converse may fail only when $v(Q)>0$
and $\bar{P}$ is a non-zero square in $k[x]$ (indeed $W$ normal
implies that if $v(Q)>0$ then $v(P)\le 1$). In this case the
transformation at step (2.b) of Algorithm~\ref{algo-normalization} provides
a pair $Q_1, P_1$ for $W$ satisfying \ref{normalize_at_2}(1.c).
\end{enumerate}  
\end{remark}

\end{section}

\begin{section}{Multiplicity \texorpdfstring{$\lambda$}{lambda}}\label{lambda}

Fix a prime $\p$ of $A$. Recall that $v=v_\p$, $k=k(\p)$ and $p$ is
the characteristic of $k$. 

\begin{definition} \label{def:lambda} 
Let $H(x)\in A[x]$, let $c\in A$. We write $H(x)$ in the form (Taylor
expansion at $c$):   
$$H(x)=\sum_i a_{c,i}(x-c)^i,$$
and define
  \begin{equation}
    \label{eq:muc}
    \mu_c(H)=\min_i \{ v(a_{c,i})+i \} =v(H(\pi x+c)).  
  \end{equation}
The map $\mu_c$ depends only on the 
class of $c$ modulo $\pi$.  We have $\mu_c(0)=+\infty$ and, if $H\ne
0$, then $\mu_c(H)$ is the biggest integer $m$ such that $H(x)\in
(x-c, \pi)^mA[x]$. In fact $\mu_c$ is the restriction to
    $A[x]$ of the Gauss valuation on $K(x)$ with respect 
to the variable $\pi x+c$. Thus we have 
$$
\mu_c(H_1H_2)=\mu_c(H_1)+\mu_c(H_2), \quad
\mu_c(H_1+H_2)\ge \min \{ \mu_c(H_1), \mu_c(H_2) \}.
$$
\end{definition}

For any pair of polynomials $Q, P\in A[x]$, denote by 
  \begin{equation}
    \lambda_c(Q,P)=\min \{ 2\mu_c(Q), \mu_c(P) \}.
  \end{equation}
  As $\mu_c$ is a valuation, when $p\ne 2$ we have 
      \begin{equation}
        \label{eq:mu_not2}
 \lambda_c(Q,P)=\lambda_c(Q, 4P)\le \mu_c(4P+Q^2) 
      \end{equation}
If we denote by $\ord_{\bar{c}}\bar{H}$ the vanishing order of
$\bar{H}(x)\in k[x]$ at $\bar{c}$, then it follows immediately from the definition
that
\begin{equation}
  \label{eq:dl}
  \ord_{\bar{c}} \bar{H}\ge \mu_c(H). 
\end{equation}
  
\begin{definition}[\cite{LTR}, D\'efinition 10, p. 4589, case $r=1$] \label{def:l}
  Let $p_0\in W_0(k)$ (see Definition~\ref{sb-W})
  be a rational point. Let $\bar{c}=x(p_0)\in k$ be the 
$x$-coordinate of $p_0$ for some $c\in A$. We define the \emph{multiplicity
$\lambda(p_0)$} by 
\begin{equation}
\label{eq:lambda2}
\lambda(p_0)=\max
\{ \lambda_c (Q-2H, P+QH-H^2) \mid
H(x)\in A[x]  \}
\end{equation}
(we have $(y+H)^2+(Q-2H)(y+H)=P+QH-H^2$.)

The \emph{multiplicity of the pole of $x$} in $W_k$ is defined as
the multiplicity at $0$ of $z^2+x^{g+1}Q(1/x)=x^{2g+2}P(1/x)$. 
\end{definition}

Note that
\begin{equation}
   \lambda(p_0)=\mu_c(4P+Q^2), \quad \text{if \ } p\ne 2.
\end{equation}
Indeed, $\lambda(p_0)\le \mu_c(4P+Q^2)$ by
    Inequality~\eqref{eq:mu_not2}, and the inverse inequality holds 
    by taking $H=Q/2$ in \eqref{eq:lambda2}.

The next lemma explains how to compute $\lambda(p_0)$ and, starting
with a suitable equation, how to find a new pair $Q, P$ such that 
$\lambda(p_0)=\lambda_c(Q, P)$. 
This is partly sketched in \cite{LTR}, bottom of page 4590. 

\begin{lemma}\label{mup} Suppose that $p=2$, $W$ is normal at $\p$,
and that the equation
  $$ y^2+Q(x)y=P(x)$$
satisfies Lemma~\ref{normalize_at_2}(1). 
\begin{enumerate}[{\rm (1)}]
   \item If $2\mu_c(Q)\le \mu_c(P)$, then
    $\lambda(p_0)=2\mu_c(Q)=\lambda_c(Q, P)$.
  \item Suppose $2\mu_c(Q)>\mu_c(P)$. 
    \begin{enumerate}[{\rm (a)}] 
    \item If $\mu_c(P)$ is odd, then
    $\lambda(p_0)=\mu_c(P)=\lambda_c(Q, P)$. 
  \item Suppose $\mu_c(P)=2r\ge 0$ is even.  Write $P(x)=\sum_i a_{c,i}(x-c)^i$.
    \begin{enumerate}
    \item 
    If $\mu_c(P)=v(a_{c,i})+i$
for some odd $i$, then $\lambda(p_0)=\mu_c(P)=\lambda_c(Q, P)$. 
\item  Otherwise, let
  $$H_0(x)=\sum_{0\le i\le \min\{r, g+1\}} e_{c,i}(x-c)^{i}\in A[x]$$
with $e_{c,i}\in \pi^{r-i}A$ such that $(\pi^{i-r}e_{c,i})^2 \equiv  (\pi^{-(2r-2i)} a_{c,2i}) \mod \pi$. 
  Then
  $$\lambda_c(Q-2H_0, P+QH_0-H_0^2) > \lambda_c(Q, P).$$   
  Moreover, $v(H_0)=0$ if and only if $r\le g+1$ and $v(a_{2r})=0$.
\end{enumerate}
\end{enumerate}
\item We have $\lambda(p_0) \le 2g+3$. Moreover, if $\lambda_c(Q,P)<\lambda(p_0)$,
then there exists a new pair $Q_0, P_0\in A[x]$ for $W$ such that 
\begin{enumerate}[{\rm (a)}]
\item $Q_0\equiv Q \mod 2$, and $P_0 - P$ is congruent to a square
in $k[x]$ modulo $(\pi, Q)$;
\item in the case $v(Q)>0$ and $v(P)=1$, we have $Q_0 \equiv Q
 \mod 2\pi$, $P_0 - P \equiv 0 \mod \pi^2$; 
      \item The pair $Q_0, P_0$ satisfies the same condition in
  Lemma~\ref{normalize_at_2}(1) as $Q, P$; 
\item $\lambda_c(Q_0, P_0)=\lambda(p_0)$.
\end{enumerate} 
\end{enumerate}
\end{lemma}

\begin{proof} 
(1) We have $Q\ne 0$ because $\mu_c(0)=+\infty$.   
Let $H\in A[x]$. If $\mu_c(Q-2H)\le \mu_c(Q)$, then
$$\lambda_c(Q-2H, P+QH-H^2)\le 2\mu_c(Q-2H)\le 2\mu_c(Q)=\lambda_c(Q, P).$$ 
Suppose $\mu_c(Q-2H)>\mu_c(Q)$. Then $\mu_c(2H)=\mu_c(Q)$, so $\mu_c(H)=\mu_c(Q)-v(2)$.
As
$$\mu_c(P)\ge 2\mu_c(Q), \ \mu_c(QH)=2\mu_c(Q)-v(2), \ \mu_c(H^2)=2\mu_c(Q)-2v(2)
$$
we have $\mu_c(P+QH-H^2)=2\mu_c(Q)-2v(2)<2\mu_c(Q-2H)$ 
and
$$ \lambda_c(Q-2H, P+QH-H^2) =2\mu_c(Q)-2v(2)< 2\mu_c(Q)=\lambda_c(Q, P).$$ 
So $\lambda(p_0)=\lambda_c(Q,P)$.
\smallskip

(2) Suppose that $\lambda(p_0)>\lambda_c(Q,P)$. So
there exists $H\in A[x]$ such that
$$\lambda_c(Q-2H, P+QH-H^2) > \lambda_c(Q, P)=\mu_c(P).$$  
We have $P+QH-H^2=P+(Q-2H)H+H^2$. This implies that
$\mu_c(P)=2\mu_c(H)=2r\in 2\mathbb N$ (hence (2.a) is proved)
and $\mu_c(P+H^2)>2r$. Then  
$$ \pi^{-2r}P(\pi x+c) \equiv -(\pi^{-r}H(\pi x+c))^2 \mod \pi.$$  
Therefore, if $\lambda(p_0)> \mu_c(P)$, then 
$\mu_c(P)=\lambda_c(Q,P)$ is only reached by terms of even degrees $a_{c,2i}(x-c)^{2i}$.
This proves (2.b.i).

(2.b.ii) Now we have $\mu_c(P)=2r$ and $\mu_c(P)< v(a_{c,i})+i$ for
all odd $i$'s. By construction we see that $\mu_c(H_0)=r$,
 $\mu_c(P-H_0^2) > 2r$, and that $v(H_0)=0$ if and only if $r\le g+1$ and $v(a_{2r})=0$.
As  $\mu_c(QH_0)>2r$ and 
$$2\mu_c(Q-2H_0) \ge \min \{ 2\mu_c(Q), 2(r+v(2)) \}> 2r,$$
we have $\lambda_c(Q-2H_0, P+QH_0-H_0^2)>\lambda_c(Q, P)$.
\smallskip

(3) In the case (2.b.ii) we let temporarily $Q_0=Q-2H_0$,
$P_0=P+QH_0-H_0^2$. We may need to modify them later.
Property (a) is satisfied by construction.
We will first prove the inequality $\lambda_c(Q, P)\le 2g+3$ and, in the case
(2.b.ii), the same inequality for $Q_0, P_0$ and the
property (b). Property (c) is a direct consequence of (a) and (b). 

Notice that from the construction, we have
$\mu_c(F)\le v(F)+\deg F$ for all $F(x)\in A[x]$.

(3.1) If $v(Q)=0$, $\mu_c(Q)\le \deg Q\le g+1$. Thus
$\lambda_c(Q, P)\le 2g+2$. In the case (2.b.ii),
$v(Q_0)=0$, and $\mu_c(Q_0)\le \deg Q_0\le g+1$ because $\deg H_0\le
g+1$ by construction.  Hence $\lambda_c(Q_0, P_0)\le 2g+2$. 

(3.2) Suppose now that $v(Q)>0$. If $v(P)=0$, then $\bar{P}(x)\notin
k[x^2]$ by hypothesis and $\mu_c(P)\le \deg P\le
2g+2$. Moreover, in the case  (2.b.ii), 
$$P_0=P+QH_0-H_0^2\equiv P-H_0^2 \not\equiv \square  \mod \pi,$$ 
and $\mu_c(P_0)\le 2g+2$ as well. 

(3.3) Suppose that $v(Q)>0$ and $v(P)=1$. Then $\mu_c(P)\le 1+
\deg P\le 2g+3$. In the case (2.b.ii), $v(H_0)>0$ because $v(a_{i})=0$ 
for all $i$. So $Q_0-Q=2H_0\equiv 0 \mod 2\pi$ and $v(P_0-P)\ge
2$. Thus $v(P_0)=1$ and $\mu_c(P_0)\le \deg P_0+1\le 2g+3$.
\smallskip

To prove (3.d), if $\lambda(p_0)=\lambda_c(Q_0, P_0)$ then we are done. Otherwise, we
repeat the same operations with $Q_0, P_0$. As the $\lambda_c$ 
increases strictly, we will end-up with a pair having $\lambda_c$
equal to $\lambda(p_0)$. 
\end{proof}

\begin{lemma}\label{mup4} Keep the assumptions and notation of
  Lemma~\ref{mup} and suppose that $\lambda(p_0)=\lambda_c(Q, P)$. Let
  $r=[\lambda(p_0)/2]$, $x_1=\pi^{-1}(x-c)$, and 
  $$
Q_1(x_1)=\pi^{-r}Q(\pi x_1+c), \quad P_1(x_1)=\pi^{-2r}P(\pi x_1+c) \in A[x_1]. 
$$
Then the equation
  \begin{equation}\label{eq:Wp0}
z^2+ Q_1(x_1)z=P_1(x_1)     
  \end{equation}
($z=y/\pi^r$) 
defines a Weierstrass model $W(p_0)$ of $C$, normal at $\p$ with the
pair $(Q_1, P_1)$ satisfying Lemma~\ref{normalize_at_2}(1). 
\end{lemma}

\begin{proof} 
We have $v(Q_1)=\mu_c(Q)-r$ and $v(P_1)=\mu_c(P)-2r$. Let us
distinguish three cases.

(a) If $\lambda(p_0)=2r=2\mu_c(Q)$, then $v(Q_1)=0$; 

(b) If $\lambda(p_0)=2r=\mu_c(P)< 2\mu_c(Q)$, then by (2.b), there exists an odd index $i_0$ such
that $v(a_{i_0})+i_0=\mu_c(P)$. This implies that $\bar{P}_1(x)$ has
a non-zero odd degree term. In particular $\bar{P}_1(x)\notin k[x^2]$;

(c) If $\lambda(p_0)=2r+1$. Then $\mu_c(Q)>r$, $\mu_c(P)=2r+1$ and 
$v(Q_1)>0$, $v(P_1)=1$.

So the pair $(Q_1, P_1)$  satisfies  Lemma~\ref{normalize_at_2}(1) and $W(p_0)$ is normal.
\end{proof}

\end{section}

\begin{section}{Minimality criterion}\label{min-crit}

We fix a prime $\p$ of $A$. 
We will assume that $W$, defined by Equation~\eqref{eq:start}, is
normal at $\p$. Moreover, if $\p$ is even, we suppose that the pair
$(Q, P)$ satisfies Lemma~\ref{normalize_at_2}(1). 
\medskip

\begin{notation} \label{epsilon}
We let $\epsilon(W)=0$ if $W_k$ is a reduced scheme,
and $\epsilon(W)=1$ otherwise.
If necessary it will be denoted by $\epsilon_\p(W)$. 
\end{notation}

Under the above conditions, we have $\epsilon(W)=\min\{ v(Q), v(P)\}$
if $\p$ is even, and $\epsilon(W)=v(F)-2[v(F)/2]$ if $\p$ is odd and
$F=4P+Q^2$.

\subsection{Dilatation}\label{dilat} Let $p_0\in W(k)$. Let $W(p_0)$ be the model
defined by Equation~\eqref{eq:Wp0}). See also  a 
more geometrical description in \cite{LTR}, D\'efinition 12, p. 4592. 
\medskip 

\begin{remark} \label{Wp0-P} Let $p_0\in W(k)$. 
  \begin{enumerate}[{\rm (a)}] 
  \item The birational map $W(p_0) \dasharrow W$ is an 
isomorphism at all primes different from $\p$. 
\item Denote by $\Delta_{W}$ the discriminant of Equation~\eqref{eq:start} 
and $\Delta_{W(p_0)}$ that Equation~\eqref{eq:Wp0}. Then 
\begin{equation}
  \label{eq:dp0}
v(\Delta_{W(p_0)})-v(\Delta_W)= 2(2g+1)(g+1-2[\lambda(p_0)/2])
\end{equation}
(\cite{LTR}, Lemme 9(a), p. 4593).
\item We have
  $$\epsilon(W(p_0))=\lambda(p_0)-2[\lambda(p_0)/2].$$ 
  (If $\p$ is even, this is contained in the proof of Lemma~\ref{mup4}).
\end{enumerate}
\end{remark}

\subsection{Minimality criterion} 

The minimality of $W$ at $\p$ can be determined by looking at the
multiplicity $\lambda(p_0)$ at rational points of $W(k)$. 

\begin{proposition}[Minimality criterion]\label{min-cond} Let $\p$ be
a prime of $A$. Suppose that $W$, defined by  
  $$ y^2+Q(x)y=P(x),$$ 
is normal at $\p$.  
\begin{enumerate}[\rm (1)] 
\item If for all $p_0\in W(k)$ we have $\lambda(p_0)\le g+1$, then
    $W$ is minimal at $\p$. The converse is true if $g$ is even. 
    \footnote{Let $g\ge 1$ be odd and let $p>2$.
Consider the equation $y^2=px^{2g+1}+p^{g+2}$ over $\Z_p$. Then $\epsilon=1$. For the point $x=y=p=0$, 
we have $\lambda=g+2$. By (2.b), this equation is minimal. But
$\lambda>g+1$. So in (1) the converse does not hold for odd $g$ in general.}
\item Suppose $g$ is odd and there exists $p_0\in W(k)$ with 
    $\lambda(p_0)\ge g +2$.
    \begin{enumerate}[{\rm (a)}]
    \item If $\lambda(p_0)\ge g+3$, 
      then $W$ is not minimal at $\p$. More precisely,
      $v(\Delta_{W(p_0)})<v(\Delta_W)$.
    \item If $\lambda(p_0)=g+2$ and $\epsilon(W)=1$, then $W$
      is minimal at $\p$.
      \item If $\lambda(p_0)=g+2$ and $\epsilon(W)=0$, then $W$
      has the same discriminant as $W(p_0)$ with
      $\epsilon(W(p_0))=1$. The
model $W(p_0)$ (hence $W$) is minimal 
      at $\p$ if and only if for all $q\in W(p_0)(k)$, we have $\lambda(q)\le g+2$.       
    \end{enumerate}
\end{enumerate}
\end{proposition}

\begin{proof} (1) and (2.b) follow from \cite{LTR}, 
  Corollaire 2, p. 4594 and Lemme 9(c), p. 4593.
 (2.a) and the  first part of (2.c) follow from the equality~\eqref{eq:dp0}.
 To finish the proof of (2.c), as $\lambda(p_0)=g+2$ is odd,
 we have $\epsilon(W(p_0))=1$ by
 Remark~\ref{Wp0-P}(3). The pole of $x$ in $W(p_0)$ has 
multiplicity $\lambda=g+1$ by {\it op.\ cit.}, Lemme 9(b), p. 4593. This
finishes the proof by (1) and (2.b). 
\end{proof}

\begin{remark}\label{bms}   
  Let us say that a multiplicty $\lambda$ is \emph{small} if $\lambda\le g+1$, or if $g$
    is odd and $\lambda=g+2$ with $\epsilon =1$ (Conditions (1) or
    (2.b) in Proposition~\ref{min-cond}). We say it is \emph{medium} 
    if $g$ is odd, $\lambda=g+2$ with $\epsilon =0$ (Condition (2.c)).
    Otherwise we say
    it is \emph{big}: $\lambda\ge g+2$ and $g$ is even or $\lambda\ge g+3$ and $g$ is odd.
    \begin{enumerate}[{\rm (1)}]
    \item 
    Proposition~\ref{min-cond} then can be rephrased as following:
    \begin{enumerate}[{\rm (i)}]
    \item     If all rational points of $W(k)$ have small
      multiplicities, then $W$ is minimal; 
    \item if there is a rational point $p_0\in W(k)$ with big
      multiplicity, then $W$ is not minimal and
      $v(\Delta_{W(p_0)})<v(\Delta_W)$; 
    \item if there is a rational point $p_0\in W(k)$ with medium multiplicity, then we work with $W(p_0)$ whose discriminant has the same valuation as $W$. But $W(p_0)$ has no 
    rational point with medium multiplicity because $\varepsilon(W(p_0))=1$. 
  \end{enumerate}
\item During the minimization process (\ref{algo-even} and \ref{algo-odd}), once we encounter a rational point  $p_0$ with big or medium multiplicity, we work with $W(p_0)$ defined
  by Equation~\eqref{eq:Wp0} in Lemma~\ref{mup4}. The points
  at $\infty$ in $W(p_0)$ corresponding to $x_1=\infty$ has
  small multiplicities ($\lambda\le g+1$). This follows from \cite{LTR},
  Lemme 9.(b), page 4593 by case-by-cas analysis. These points at 
  infinity are denoted by $p_0'$ in {\it loc. cit.} 
Therefore in the next loops of the algorithm we do not have to
deal with the points at infinity.
\item \label{big-l} 
Using \cite{LTR}, Lemme 7(f), pages 4589-4590,  one can
show  that if there is more than one point in
$W(k)$ with big or medium multiplicities, then there are exactly $2$
such points $p_0, p_1$. Moreover, 
$g$ must be even, $\epsilon(W)=1$, $\lambda(p_i)=g+2$ and the
$W(p_i)$'s are then minimal at $\p$. 
\end{enumerate}
\end{remark} 
\medskip

The minimality criterion ~\ref{min-cond} needs {\it a priori} to compute the
multiplicity $\lambda$ for all points in $W(k)$. The next lemma
explains that it is only necessary to do it  for at most 2 points
of $W(k)$ and how to find them.  

\begin{lemma} \label{l-d} Keep the notation of the above proposition.
Suppose further that when $\p$ is even, $(Q, P)$ satisfies 
Lemma~\ref{normalize_at_2}(1). 
Let $p_0\in W_0(k)$ be such that $\lambda(p_0)\ge g+2$ and denote by 
$\epsilon=\epsilon(W)$. Let $\bar{c}=x(p_0)$. 
  \begin{enumerate}[{\rm (1)}]
\item\label{l-d-odd} If $p\ne 2$, then 
  $$\ord_{\bar{c}}(\overline{\pi^{-\epsilon}(4P+Q^2)})\ge g+2-\epsilon.$$
\item \label{l-d-even} Suppose $p=2$. 
  \begin{enumerate}[{\rm (i)}]
  \item  If $\bar{Q}\ne 0$,    then 
  $$\ord_{\bar{c}}(\overline{Q}) \ge (g+2)/2. $$
\item If $\bar{Q}=0$ and $\bar{P}\notin k[x^2]$, then 
  $$\ord_{\bar{c}}(\overline{P'}) \ge g+1,$$
  where $P'$ is the derivative of $P(x)$. 
 \item If $\epsilon=1$, then 
 $$\ord_{\bar{c}}(\overline{\pi^{-1}Q})\ge g/2,
    \quad \ord_{\bar{c}}(\overline{\pi^{-1}P})\ge
    g+1.$$
\end{enumerate}
\end{enumerate}
\end{lemma}

\begin{proof} (1) follows from Inequality~\eqref{eq:dl}.
  \smallskip
  
  (2) Suppose $\lambda_c(Q,P)<\lambda(p_0)$. Then we are in the the case (2.b.ii) of Lemma~\ref{mup}.  
Let $(Q_0, P_0)$ be the pair given by Lemma~\ref{mup}(3). Then the
properties (3.a)-(3.b) there imply immediately that in the computations
of the vanishing orders we can replace $(Q, P)$ by $(Q_0, P_0)$.
Therefore we can suppose $\lambda_c(Q, P)=\lambda(p_0)$. 
Again by Inequality~\eqref{eq:dl} 
$$2\ord_{\bar{c}}(\bar{Q})\ge 2\mu_c(Q)\ge g+2,
\quad \ord_{\bar{c}}(\bar{P})\ge \mu_c(P)\ge g+2.$$
This proves (i) and (ii). When $\epsilon=1$, the same proof works by noting that
$\mu_c(F)=\mu_c(\pi F)-1$ for any $F\in A[x]$. 
\end{proof}

The next proposition is just a more explicit transcription of
\cite{LTR}, Corollaire 2(b), p. 4594, plus Lemme 9(c), p. 4593 for (2.a)). The parenthetical sentence in (2.b) below
follows from {\it op.\ cit.}, Lemme 7(f), p. 4595-4596.
These results rely on the invariant denoted by
    $\lambda'(p_0)$ which is the maximum of the $\lambda(p)$'s for all
  rational points $p\in W(p_0)(k)$ except the points at infinity
  ($x_1=\infty$ in the notation of Lemma~\ref{mup4}).

\begin{proposition}[Uniqueness criterion] \label{min-uni} 
 The following properties are true. 
  \begin{enumerate}[\rm (1)]
  \item If for all $p_0\in W(k)$ we have $\lambda(p_0)\le g$, 
    then $W$ is the unique minimal model at $\p$.
  The converse is true if $g$ is odd. 
\item Suppose $g$ is even and there exists $p_0\in W(k)$ with
    $\lambda(p_0)=g+1$. 
 \begin{enumerate}[{\rm (a)}]
 \item If $\epsilon(W)=1$, {\bf and if $W$ is minimal\footnote{this
       hypothesis was omitted in the published version.}}, then $W$ is the unique minimal model at $\p$.
 \item Suppose $\epsilon(W)=0$, then $W$ is the unique minimal model at $\p$
   if and only if for all $p_0\in W(k)$ with $\lambda(p_0)=g+1$
   (there are at most two such points), we have $\lambda(q)\le g+1$
   for all $q\in W(p_0)(k)$. 
\end{enumerate}
\end{enumerate}
\end{proposition}

The next lemma allows to construct a global equation from local
equations at odd primes and even primes. 

\begin{lemma}[Combining local equations] \label{odd-2-even} Suppose $\mathrm{char}(K)=0$. Let
\begin{equation}\label{eq:start2} 
y^2+Q(x)y=P(x)
\end{equation} 
be an equation of $C$ over $A$ with $\deg Q\le g+1$ and $\deg P\le
2g+2$, and let    
\begin{equation}\label{eq:odd} 
z^2=F_1(x_1)
\end{equation} 
be another equation of $C$ over $A$ obtained by the change of variables
$$x=\frac{ax_1+b}{cx_1+d}, \quad 2y+Q(x)=\frac{ez}{(cx_1+d)^{g+1}}$$
with coefficients in $A$ and such that $e, ad-bc$ are odd.
Let $m\in A$ be such that
$$ 4m \equiv 1 \mod e $$ 
(take $m=0$ if $e\in A^*$). Let
$$ y=\frac{ey_1}{(cx_1+d)^{g+1}}, $$
$$ Q_1(x_1)=e^{-1}(1-4m)(cx_1+d)^{g+1}Q(x) $$
and
$$ P_1(x_1)=e^{-2}(cx_1+d)^{2g+2}(P(x)+(2m-4m^2)Q(x)^2). $$
Then 
\begin{equation}
  \label{eq:new}
  y_1^2+Q_1(x_1)y_1=P_1(x_1) 
\end{equation} 
is an equation of $C$ over $A$, equivalent 
to Equation~\eqref{eq:start2}  (Definition~\ref{sb-W})
at even primes and to Equation~\eqref{eq:odd} at odd primes.  
\end{lemma}

\begin{proof}   It is clear that $Q_1(x_1)\in A[x_1]$. We have 
  $$4P_1(x_1)+Q_1(x_1)^2=F_1(x_1),$$
  so $4P_1(x_1)\in A[x_1]$. But by
construction $P_1(x_1)\in A_e[x_1]$,
where $A_e$ is the localization of $A$ with respect to
      the positive powers of $e$, thus $P_1(x_1)\in A[x_1]$.
Moreover the above relation implies that Equation~\eqref{eq:new} is
equivalent to Equation~\eqref{eq:odd} at odd primes. As $e$ and
$ad-bc$ are odd, hence invertible at even primes, 
Equation~\eqref{eq:new} is equivalent to Equation~\eqref{eq:start2} at
even primes. 
\end{proof}

\end{section}

\begin{section}{Pointed minimal Weierstrass  equations} \label{pointed-min} 

  Suppose that $C$ has a rational Weierstrass point $w_0$.
A \emph{pointed Weierstrass equation} of $(C, w_0)$ over $A$ is an equation
$$
y^2+Q(x)y=P(x) 
$$ 
over $A$ with $\deg Q\le g$ and $P(x)$ monic of $\deg P=2g+1$, and $w_0$ is the pole
of $x$. By suitably scaling an initial equation of $C$ over $K$ with
$x$ having its pole at $w_0$, one can always obtain such an equation. The
associated model is automatically normal with reduced fiber at all primes $\p$ 
(see \S~\ref{normalization}).

\emph{Minimal pointed 
Weierstrass equations of $C$} are defined in a similary way to the
non-pointed case. 
They are studied in \cite{Loc}. 
For a given $(C, w_0)$, the minimal pointed Weierstrass equation 
exists and is unique up to the transformations described in Lemma~\ref{p-c}
below. See also \cite{Loc} and \cite{LGE}, Corollary 5.2.
The next lemma is stated in \cite{Loc}, Remark after Definition 2.1.

\begin{lemma}\label{p-c} Fix a prime $\p$ of $A$. Let
$R=A_\p$.  Let
$$ y_1^2+Q_1(x_1)y_1=P_1(x_1) $$
be another pointed Weierstrass equation of $(C, w_0)$ with
discriminant $\Delta_1$ such that $v_\p(\Delta_1)\le v_\p(\Delta)$. Then
there exist $u, c\in R$, $H(x)\in R[x]$ of degree $\le g$ such that 
$$x=u^2 x_1 + c,  \ y=u^{2g+1}y_1+H(x),$$
and we have $\Delta_1=u^{-4g(2g+1)}\Delta$. 
\end{lemma}

\begin{proof} We only have to complete the proof of the
integrality statement: $c\in R, H(x)\in R[x]$ 
(left to the reader in \cite{Loc}).  Suppose that $c\notin R$. Then
$c^{-1}\in \pi R$. As 
$$
\frac{1}{x}-c^{-1}=-(c^{-1}u)^2 x_1', \quad \text{where }
x_1'=\frac{x_1}{(c^{-1}u^2)x_1+1}, 
$$
by \cite{LGE}, Lemma 5.1, we have $v(\Delta)< v(\Delta_1)$ (with the notation
of {\it op.\ cit.}, $d=2v(c^{-1}u)>0$ and the point $p$ is the pole of $x$
in $W_k$ which is a smooth point). So $c\in R$.

It remains to prove that $H(x)\in R[x]$ and $\deg H\le g$. Without loss of
generality we can suppose that $c=0$. We then have
$$ P(x)+Q(x)H(x)+H(x)^2=u^{2(2g+1)}P_1(x_1)=u^{2(2g+1)}P_1(x/u^2)\in R[x].$$
This implies that $H(x)\in R[x]$ using the Gauss valuation on $K(x)$
associated to $v_\p$. Finally, the same equality implies that $\deg
H\le g$. 
\end{proof}

\begin{corollary} Let $y^2+Q(x)y=P(x)$ be a pointed Weierstrass equation of $(C, w_0)$.
Let $W$ be the associated Weierstrass model. Then the equation is not 
pointed-minimal at $\p$ if and only if
there exist $p_0\in W_0(k)$ such that $\lambda(p_0)=2g+1$ and
$q_0\in W(p_0)(k)$ such that $\lambda(q_0)=2g+2$.
The model $W(p_0)(q_0)$ is then a pointed model and
$$v(\Delta_{W(p_0)(q_0)})=v(\Delta)-4g(2g+1).$$ 
\end{corollary}

\begin{proof} Suppose that the equation is not pointed-minimal at
$\p$. By Lemma~\ref{p-c}, there exists a pointed equation over $A_\p$
$$
y_1^2+Q_1(x_1)y_1=P_1(x_1) 
$$
with $x=\pi^{2r} x_1+c$ for some $r\ge 1$ and 
$y=\pi^{r(2g+1)}y_1+H(x)$. Approximating  elements of $A_\p$ by that 
of $A$, one can suppose that all coefficients belong to $A$. 
Indeed, let $a \in A, U(x)\in A[x]$
be such that $$c-a \in \pi^{2r}A_\p, \quad H(x)-U(x)\in \pi^{r(2g+1)}
A_\p[x].$$
Consider the change of variables
$x_2=x_1+\pi^{-2r}(c-a)$ and $y_2=y_1+ \pi^{-r(2g+1)}(H-U)$ over
$A_\p$. We then have $x-\pi^{2r}x_2=a\in A$ and
$y-\pi^{r(2g+1)}y_2=U(x)\in A[x]$.

Translating $x$ by $c$ and replacing $y$
with $y-H(x)$, we can suppose that $c=0$ and $y= \pi^{r(2g+1)}y_1$.
This implies that
\begin{equation}
  \label{eq:QrPr}
  Q(x)=\pi^{r(2g+1)}Q_1(x/\pi^{2r}), \quad 
P(x)=\pi^{2r(2g+1)}P_1(x/\pi^{2r}).
\end{equation}
So $\mu_0(Q)\ge g+1$ and $\mu_0(P)=2g+1$ is odd. Therefore
if $p_0\in W(k)$ is the zero of $x$ (meaning that $p_0$
  is the unique point whose $x$-coordinate is zero), then
$\lambda(p_0)=2g+1$ (Lemma~\ref{mup}.(2.a)). An equation of $W(p_0)$ is 
then
$$
y_0^2+ \pi^{-g} Q(\pi x_0)  y_0 = \pi^{-2g} P(\pi x_0)
$$
with $x=\pi x_0$ and $y=\pi^{g}y_0$. Using the relations
\eqref{eq:QrPr}, we see that as elements of $A[x_0]$, we have
$\mu_0(\pi^{-g}Q(\pi x_0))\ge g+1$, and $\mu_0(\pi^{-2g}P(\pi x_0))=2g+2$
is reached at the odd degree $i=2g+1$.
Let $q_0\in W(p_0)(k)$ be the zero of $x_0$. 
Then $\lambda(q_0)=2g+2$ by  Lemma~\ref{mup}, (1) and (2.b.i). 
\medskip

Let us prove the converse. Under the hypothesis of the lemma, it is
enough to show that $W(p_0)(q_0)$ is defined by a pointed Weierstrass
equation. Its discriminant is given by Remark~\ref{Wp0-P}(2). 
Let $\bar{c}$ be the $x$-coordinate of
$p_0$. If $\min\{ 2\mu_c(Q), \mu_c(P)\}< \lambda(p_0)$,
we modify $Q, P$ as in Lemma~\ref{mup}(2.b.ii).
As the coefficient of degree $2g+2$ of $P(x)$ is
zero, we see by construction that $\deg H_0\le g$. Repeating the
algorithm if necessary, we get a new pair $(Q, P)$ such that
$\min\{ 2\mu_c(Q), \mu_c(P)\}=\lambda(p_0)$, $\deg Q\le g$
and $P(x)$ is monic of degree $2g+1$.

An equation of $W(p_0)$ is 
$$
y_0^2+ \pi Q_0(x_0)   y_0 = \pi P_0(x_0) 
$$
with
$$Q_0(x_0)=\pi^{-g-1} Q(\pi x_0+c), \quad 
P_0(x_0)=\pi^{-2g-1} P(\pi x_0+c)$$
and $P_0(x_0)$ is monic. 
Let $\bar{c}_1$ be the $x_0$-coordinate of $q_0$. 
Again, if 
$\lambda(q_0)$ is not reached by the pair $\pi Q_0, \pi P_0$, we
see in the construction of Lemma~\ref{mup}(2.b.ii) that $H_0(x_0)$ is divisible by $\pi$ and
has degree $\le g$. So we get a new pair $\pi Q_1(x_0), \pi P_1(x_0)$
reaching $\lambda(q_0)$ and
such that $\deg Q_1\le g$ and $P_1(x_0)$ is monic of degree $2g+1$. 
Then $W(p_0)(q_0)$ is defined by the equation 
$$
y_1^2+ \pi^{-g} Q_1(\pi x_1+c_1)   y_0 = \pi^{-2g+1} P_1(\pi x_1+c_1) 
$$
which is a pointed Weierstrass equation. 
\end{proof}

\begin{remark} Similarly to Lemma~\ref{mup}(3), because $\deg Q\le g$
  and $P(x)$ is monic of degree $2g+1$, one can show that
  for all $p_0\in W_0(k)$ (resp. $q_0\in W(p_0)(q_0)(k)$),
  $\lambda(p_0)\le 2g+1$ (resp. $\lambda(q_0)\le 2g+2$). Moreover, by
  Remark~\ref{bms}.\ref{big-l}, there is at most one such point.
\end{remark}

\begin{remark}
  Suppose $\p$ is odd. It defines an absolute value $|. |_\p$ on $K$.
      Fix an algebraic closure $K^{\mathrm{alg}}$ of $K$ and an extension of
      $|. |_\p$ to it.  
  Then a pointed Weierstrass equation
  $y^2=P(x)$ is pointed-minimal at $\p$ if and only if
  there is no disc centered  in $A$ of radius
  $\le |\pi|_\p^2$ containing
  all roots of $P(x)$ in $K^{\mathrm{alg}}$.

Indeed, if such a disc, centered in some $c\in A$ exists, then the change of variables $x=\pi^2 x_1+c$, $y=\pi^{2g+1} y_1$, leads to a pointed Weierstrass equation of discriminant $\pi^{-4g(2g+1)} \Delta$. Conversely, if $y^2=P(x)$ is not pointed-minimal, then the minimal one is given by a change of variables
    as in Lemma~\ref{p-c} with $H(x)=0$ and $v(u)>0$. As $A$ is dense in $A_\p$, one can take $c\in A$. Translating $x$ by $c$ we can suppose that $x=u^2x_1$. As $P(u^2x_1)\in u^{2(2g+1)}A_\p[x_1]$. This implies that the roots
$\alpha\in K^{\mathrm{alg}}$ of $P(x)$ all have $|\alpha|_\p\le |u|_\p^2\le |\pi|_\p^2$. 
  \end{remark}

\end{section}

\begin{section}{Minimization algorithm} \label{mini-algo}

We start with a Weierstrass equation
\begin{equation}\label{eq:start3} 
  y^2+Q(x)y=P(x)
\end{equation}
of $C$ over $A$, with discriminant $\Delta$ and
$\deg Q\le g+1$, $\deg P\le 2g+2$.
Note that the formula \eqref{eq:d1} implies that 
if $v_\p(\Delta)<2(2g+1)$ 
(resp. $v_\p(\Delta)<4(2g+1)$) if $g$ is even (resp. if $g$ is odd),
then the equation is minimal at $\p$.  

Let us describe the minimization algorithm.
It consists in minimizing successively at even primes dividing $\Delta$,
then at odd primes dividing $\Delta$, and finally we globalize in
\S~\ref{final-s} using Lemma~\ref{odd-2-even}. 
The local minimization is done  step by step as follows. Start with an integral Weierstrass equation and a prime
$\p$ dividing $\Delta$. 
\begin{enumerate}[\rm (i)]
\item First normalize the equation at $\p$ (Algorithm~\ref{algo-normalization}).
\item Candidates for rational points aways from $x=\infty$ having big or medium multiplicities 
are found by Algorithm~\ref{algo-find-c-even} and \ref{algo-find-c-odd}.
For even primes, whether the points at infinity might have big or medium
multiplicities are checked in Test~\ref{algo-c-OK}
(after inverting $x$ to reduce to the points with $x=0$). 
For odd primes this task is much simpler and is included directly in
\ref{algo-odd}(\ref{algo-I}).
\item For a candidate $p_0$ found above, the actual multiplicity is then computed in Algorithm~\ref{algo-lambda} for even
primes (no need of algorithme at odd
primes). Test~\ref{algo-lambda-OK} tells us whether $\lambda$ is small
or not. If yes we move to the next candidate. 
\item  As soon as a rational point $p_0$ with big $\lambda$ is found, we consider a new
normal Weierstrass equation, corresponding to $W(p_0)$ (Lemma~\ref{mup4}). Then we start
again the algorithm with this new equation. In the new equation, the
points at infinity always have small multiplicities.

We have $v(\Delta_{W(p_0)})\le v(\Delta_W)$, with equality if and only
$\lambda(p_0)$ is medium. In this case, either $W(p_0)$
(and hence $W$) is minimal, or there exists a rational point
$p_1\in W(p_0)(k)$ with big multiplicity, so that after the next change of
variables, the valuation of the discriminant will decrease strictly. After finitely
many loops we get in Step~(\ref{p-end}). 
\item\label{p-end} If there is no point with big or medium
  multiplicity, then the equation is minimal at $\p$.  
\end{enumerate}

\subsection{Normalization algorithm}  

\begin{algo}[Lemma~\ref{normalize_at_2}]\label{algo-normalization}{\hskip 4pt}

{\bf Input} : an even prime $\p$ of $A$, a pair of polynomials $Q, P$ as in
Equation~\eqref{eq:start3}, and $e_0\in A$
(corresponding to the $e$ in Equation~\eqref{eq:cvar} for even primes.) 
\smallskip
    
{\bf Output}: a new pair $Q, P$ defining a Weierstrass model normal at
$\p$ satisfying Lemma~\ref{normalize_at_2}(1), 
and the new $e_0$. 
\medskip

\begin{enumerate} 
  \item If $v(Q)=0$, go to (3). 
  \item Otherwise, 
    \begin{enumerate} 
    \item if $v(P)=0$, and $\bar{P}\notin k[x^2]$, go to (3); 
    \item if $v(P)=0$ and $\bar{P}(x)\in k[x^2]$. Let
      $H(x)=\sum_{0\le i\le g+1} b_ix^i$ with
      $\bar{b}_i^2=\bar{a}_{2i}$ ($k$ is perfect of characteristic $2$). Then
            $$Q\gets Q-2H, \quad P\gets P+QH-H^2.$$
            (The new pair satisfies $v(Q), v(P)>0$.)
     \item if $v(P)=1$, go to (3); 
    \item if $v(P)\ge 2$, let $r=[\frac{1}{2}\min \{ 2v(Q), v(P)\}]$. 
      Then
      $$e_0\gets \pi^{r}e_0, \quad Q\gets \pi^{-r}Q, \quad P\gets \pi^{-2r}P.$$
Restart at (1).  
\end{enumerate}
\item Output $Q, P, e_0$ 
  and $\epsilon_\p=\min\{ v(Q), v(P)\}$.
  \end{enumerate}
\end{algo}

The algorithm terminates because when we need to restart the loop
(only at the step (2.d)),  the discriminant of the new equation is
equal to the previous one divided by $\pi^{4r(2g+1)}$ with $r>0$.
\medskip

Recall that at odd primes, the normalization of $y^2=F(x)$ consists just in dividing
both sides by a biggest odd square of $\mathrm{cont}(F)$.

\subsection{Computing the multiplicity \texorpdfstring{$\lambda$}{lambda}}
Let $\p$ be a prime of $A$, let $p_0=(\bar{c}, \bar{d})$ be a solution
of Equation~\eqref{eq:start3} mod $\p$ with $c, d\in A$. 
We want to compute $\lambda(p_0)$ (see Definition~\ref{def:l}). See also
Definition~\ref{def:lambda} for the notation $\mu_c$.
Recall that $\lambda(p_0)=\mu_c(4P+Q^2)$ if $\p$ is odd. 

\begin{algo}[Lemma~\ref{mup}] \label{algo-lambda} {\hskip 4pt}

  {\bf Input:} an even prime $\p$, a pair $Q, P$ satisfying
Lemma~\ref{normalize_at_2}(1) at $\p$ and $\bar{c}\in k$. 
 \smallskip

 {\bf Output:} the multiplicity $\lambda(p_0)$ and new pair $Q, P$
 such that $\lambda(p_0)=\min\{ 2\mu_c(Q), \mu_c(P)\}$. 
\medskip

\begin{enumerate}
\item Compute $\mu_c(Q), \mu_c(P)$;  
\item if $2\mu_c(Q)\le \mu_c(P)$, then $\lambda(p_0)=2\mu_c(Q)$, go to (\ref{end-lambda}); 
\item if $\mu_c(P)$ is odd, then  $\lambda(p_0)=\mu_c(P)$, go to (\ref{end-lambda}); 
 \item write $P(x)=\sum_i a_{c,i}(x-c)^i$. If $\mu_c(P)=v(a_{c,i})+i$
for some odd $i$, then $\lambda(p_0)=\mu_c(P)$, go to (\ref{end-lambda}); 
\item set $H_0(x)=\sum_{i\le r} e_{c,i}(x-c)^{i}$ where the
  sum runs through the indexes such that $v(a_{c,2i})+2i=2r$ and where $e_{c,i}\in \pi^{r-i}A$ satisfy 
  $$(\pi^{i-r}e_{c,i})^2 + \pi^{2i-2r}a_{c,2i}\equiv 0 \mod \pi.$$
Then
$$Q\gets Q-2H_0,  \quad P\gets P+QH_0-H_0^2.$$
Go back to (1).
\item \label{end-lambda} Output $Q, P$ and $\lambda=\lambda(p_0)$. 
\end{enumerate}
\end{algo}
\medskip

This algorithm does not change $\epsilon_\p$ and the conditions
in Lemma~\ref{normalize_at_2}(1). 
The test below will say if the multiplicity $\lambda(p_0)$ is small
({\tt true}) or too big ({\tt false}) in which case our equation is
(probably) not minimal at $\p$. 

\begin{test}[Remark~\ref{bms}] \label{algo-lambda-OK} \hskip 4pt 

  {\bf Input:} a multiplicity $\lambda$ and $\epsilon_\p$.
  \smallskip

  {\bf Output:} {\tt true} or {\tt false}.
  
\begin{enumerate}
\item If $\lambda\le g+1$, output {\tt true} ($\lambda$ is small);
\item if $g$ is even, output {\tt false} ($\lambda$ is big);
\item if $\lambda\ge g +3$, output {\tt false} ($\lambda$ is big);
\item if $\epsilon_\p=1$, output {\tt true} ($\lambda$ is small);
\item output {\tt false} ($\lambda$ is medium).
\end{enumerate}
\end{test}

\subsection{Finding \texorpdfstring{$p_0$}{p0} with big or medium \texorpdfstring{$\lambda$}{lambda}} To use the minimality criterion
(Proposition~\ref{min-cond}), we only have to compute $\lambda$ for
the $\bar{c}$'s given by the algorithms below.

\begin{algo}[Lemma~\ref{l-d}(\ref{l-d-even})] \label{algo-find-c-even}  
  \hskip 4pt 

  {\bf Input:} An even prime $\p$ of $A$, a pair $Q, P$ satisfying
Lemma~\ref{normalize_at_2}(1). 
\smallskip

{\bf Output:} $\epsilon_\p$ and the elements $\bar{c}\in k$ such that the
corresponding $\lambda(p_0)$ may not be small. There are at most two such
$\bar{c}$'s. 
\smallskip
  
\begin{enumerate}
\item Output  $\epsilon_\p=\min\{v(Q), v(P)\}$.
\item If $\bar{Q}\ne 0$, output the $\bar{c}\in k$
  such that
    $$\ord_{\bar{c}}(\overline{Q}) \ge (g+2)/2.$$
  \item If $\bar{Q}=0$ and $\bar{P}\notin k[x^2]$,
   output the $\bar{c}\in k$
  such that 
  $$ \ord_{\bar{c}}(\overline{P'}) \ge  g+1$$
where $P'$ is the derivative of $P$. 
\item If $\epsilon_\p=1$ (so $\bar{Q}=\bar{P}=0$), output the $\bar{c}\in k$ such that 
 $$\ord_{\bar{c}}(\overline{\pi^{-1}Q})\ge g/2
\quad \text{and }  \ord_{\bar{c}}(\overline{\pi^{-1}P})\ge
    g+1.$$      
\end{enumerate}
\end{algo}
\medskip

\begin{algo}[Lemma~\ref{l-d}(\ref{l-d-odd})] \label{algo-find-c-odd}  \hskip 4pt

{\bf Input:} An odd prime $\p$ of $A$ and $F\in A[x]$ such that $v(F)\le 1$. 
\smallskip

{\bf Output:} $\epsilon_\p$ and the elements $\bar{c}\in k$ such that the
corresponding $\lambda(p_0)$ may not be small. There are at most two such
$\bar{c}$'s. 
\smallskip

  \begin{enumerate}
\item Output $\epsilon_\p=v(F)$.  
\item Output the zeros in $k$ of $\overline{\pi^{-\epsilon_\p}F(x)}$
of order $\ge g+2-\epsilon_\p$. 
\end{enumerate}
\end{algo}
\medskip

\begin{test}[Lemma~\ref{l-d}(\ref{l-d-even})] \label{algo-c-OK}
  \hskip 4pt 
  
  {\bf Input:} same as in Algorithm~\ref{algo-find-c-even}. 
  \smallskip

  {\bf Output:} {\tt true}, if $0$ belongs to the list returned by
  Algorithm~\ref{algo-find-c-even}, and {\tt false} otherwise. 

  \begin{enumerate}
\item If $\bar{Q}\ne 0$, output {\tt true} if 
  $$\ord_{0}(\overline{Q}) \ge (g+2)/2.$$
  and {\tt false} otherwise. 
\item If $\bar{Q}=0$ and $\bar{P}\notin k[x^2]$,
  output {\tt true} if
  $$ \ord_{0}(\overline{P'}) \ge  g+1,$$
and {\tt false} otherwise. 
\item If $\epsilon_\p=1$, output {\tt true} if 
 $$\ord_{0}(\overline{\pi^{-1}Q})\ge g/2 
    \quad \text{and } \ord_{0}(\overline{\pi^{-1}P})\ge     g+1$$
and {\tt false} otherwise.       
\end{enumerate}
\end{test}

\subsection{Minimization at even primes}\label{minimizing}

\begin{algo} \label{algo-even} \hskip 4pt

  {\bf Input:} Equation~\eqref{eq:start3}.
  \smallskip

  {\bf Output:} A new equation minimal at even primes
together with
  the change of variables on $x$ given by a matrice $M_0\in
\mathrm{M}_{2\times 2}(A)$, and the multiplicative factor $e_0\in A$
in the change of variables on $y$. 
\smallskip

The algorithm will produce a new equation $y_0^2+Q_0(x_0)y_0=P_0(x_0)$ with $x=M_0x_{0}$,
$y=(e_0y_{0}+H(x_{0}))/(c_0x_{0}+d_0)^{g+1}$, where $(c_0, d_0)$ is
the second row of $M_0$,  for some $H(x_0)\in A[x_0]$.
\smallskip
  
Let $\p_1, \dots, \p_m$ be the even primes dividing the discriminant
$\Delta$ of Equation~\eqref{eq:start3}. 
We start with $i=1$, $M_0=I_2\in \mathrm{Gl}_2(A)$, $e_0=1$. 
\medskip

\begin{enumerate}[{\rm (I)}]
\item \label{algo2-0} Let $\p=\p_i$. Run
  Algorithm~\ref{algo-normalization}. If
      $v(\Delta)<2(2g+1)$ (for even $g$) or $v(\Delta)<4(2g+1)$ (for
      odd $g$), goto (\ref{algo2-III}).
\item\label{algo2-I}   Let $Q_\infty(x)=x^{g+1}Q(1/x)$ and
  $P_\infty(x)=x^{2g+2}P(1/x)$. Run Test~\ref{algo-c-OK} with the pair $Q_\infty(x), P_\infty(x)$. 
\begin{enumerate}
\item  Go to (\ref{algo2-II}) if we get {\tt false}. 
\item Otherwise, 
run Algorithm~\ref{algo-lambda} for the pair $Q_\infty(x), P_\infty(x)$ at $\bar{c}=0$.
\item Run Test~\ref{algo-lambda-OK}. If we get {\tt true},
go to (\ref{algo2-II}).  
\item Otherwise, set $r=[\lambda/2]$,
$$ Q(x)\gets \pi^{-r}Q_\infty(\pi x), \quad    P(x) \gets \pi^{-2r}P_\infty(\pi x).$$
where $Q_\infty(x), P_\infty(x)$ are the new polynomials given at (b), 
$$ e_0\gets e_0\pi^{r},   \quad 
M_0 \gets M_0\begin{pmatrix} 
    0 & 1\\ \pi & 0
  \end{pmatrix}.
$$ 
Go to (\ref{algo2-II}). 
\end{enumerate}
\item\label{algo2-II} 
Run Algorithm~\ref{algo-find-c-even}. 
\begin{enumerate}
\item Pick the first $\bar{c}$. Run Algorithm~\ref{algo-lambda} for this $\bar{c}$.
Run Test~\ref{algo-lambda-OK}. If we get {\tt true}, 
go back to (a) with the next $\bar{c}$. If there is no 
$\bar{c}$ left, go to (\ref{algo2-III}).  
\item As soon as we get {\tt false} for some $\bar{c}$,
  set $r=[\lambda/2]$. Then 
  $$
Q(x)\gets \pi^{-r}Q(\pi x+c), \quad P(x)\gets \pi^{-2r}P(\pi x + c). 
  $$
$$ e_0\gets e_0\pi^{r},    \quad
    M_0 \gets M_0\begin{pmatrix} 
    \pi & c\\ 0 & 1
  \end{pmatrix}.
$$ 
Go back to (\ref{algo2-II}).
  \end{enumerate}
\item \label{algo2-III} If $i<m$, then $i\gets i+1$ and 
go back to (\ref{algo2-0}). Otherwise output $Q, P$, $e_0$ and $M_0$. 
\end{enumerate}
\end{algo}

\subsection{Minimization at odd primes}\label{minimizing-O}
Suppose $A$ has both even and odd primes. Let 
$$F(x)=4P(x)+Q(x)^2$$  
where $P, Q$ are the polynomials
returned by Algorithm~\ref{algo-even}. 
\medskip

\begin{algo} \label{algo-odd} \hskip 4pt 

  {\bf Input:} $F(x)$ as above.
\smallskip

{\bf Output:} An equation minimal at all odd primes of $A$, together
with the changes of variables on $x$ and $y$ leading to the new equation.
\medskip

Number the odd primes $\p_1, \dots, \p_n$ dividing
the discriminant $\Delta$ of Equation~\eqref{eq:start3}.
Start with $M_1=I_2\in \mathrm{Gl}_2(A)$, $e_1=1$ and $i=1$.

\begin{enumerate}[{\rm (I)}] 
\item{(Normalization at all odd primes)} Let $s^2$ be a greatest odd square dividing $\mathrm{cont}(F)$. 
  $$ F(x) \gets s^{-2}F(x), \quad e_1\gets e_1s. $$ 
\item\label{algo-I} Let $\p=\p_i$. If
      $v(\Delta)<2(2g+1)$ (for even $g$) or $v(\Delta)<4(2g+1)$ (for
      odd $g$), goto (\ref{algo2-III}). Otherwise, let $\epsilon_\p=v(\mathrm{cont}(F))$.
  If $$\deg \overline{\pi^{-\epsilon_\p}F(x)}\ge g+1+
  \epsilon_\p,$$
go to (\ref{algo-II}). Otherwise, compute $\lambda:=\mu_0(x^{2g+2}F(1/x))$. 
  \begin{enumerate}
\item Run Test~\ref{algo-lambda-OK}. If we get {\tt true},
    go to (\ref{algo-II}).  
  \item Otherwise, set $r=[\lambda/2]$. Then
    $$F(x)\gets \pi^{-2r} (\pi x)^{2g+2}F(1/(\pi x))$$ 
    $$e_1\gets \pi^{r}e_1, \quad M_1 \gets M_1\begin{pmatrix}
    0 & 1\\ \pi & 0
  \end{pmatrix}.
    $$
  Go to (\ref{algo-II}). 
  \end{enumerate}
\item  \label{algo-II}  Run Algorithm~\ref{algo-find-c-odd}.
 \begin{enumerate}
\item Pick the first $\bar{c}$ and compute $\lambda:=\mu_c(F)$.
Run Test~\ref{algo-lambda-OK}. If we get {\tt true},
pass to the next $\bar{c}$ and restart at (a). If there is no $\bar{c}$ left, go to (\ref{algo-III}). 
\item Otherwise we get {\tt false} for some $\bar{c}$. 
Set $r=[\lambda/2]$. Then 
   $$F(x)\gets \pi^{-2r} F(\pi x+c)$$ 
    $$e_1\gets \pi^{r}e_1, \quad M_1 \gets M_1 \begin{pmatrix}
    \pi &  c \\
    0 & 1 
  \end{pmatrix}.
    $$
    Go back to (\ref{algo-II}).
\end{enumerate}
\item \label{algo-III} If $i<n$, then $i\gets i+1$ and go back to
  (\ref{algo-I}). Otherwise output $M_1$, $e_1$ and $F(x)$. 
\end{enumerate}
\end{algo}
\medskip

\subsection{Final step}\label{final-s} We give a minimal equation of $C$ over $A$.

\begin{enumerate}
\item If there are only even primes dividing the initial $\Delta$,
  Algorithm~\ref{algo-even} already returned a global minimal
  equation of $C$ over $A$.
\item Suppose there are odd primes dividing the initial $\Delta$. Let
  $Q_0, P_0$ denote the pair $Q, P$ returned by Algorithm~\ref{algo-even} and let $e_1\in A$ and
$M_1={\begin{pmatrix} a_1 & b_1 \\ c_1 & d_1 \end{pmatrix}}\in
\mathrm{M}_{2\times 2}(A)$ be returned
by Algorithm~\ref{algo-odd}. Note that $e_1$ and $\det M_1$ are odd, 
and $e_1\notin A^*$. Let $m\in A$ be such that
\begin{equation}
  \label{eq:m}
  4m \equiv 1 \mod e_1. 
\end{equation} 
For computations, $m$ should be chosen as small as
    possible, for whatever measure of the size. Let 
$$ Q_1(x_1)=e_1^{-1}(1-4m)(c_1x_1+d_1)^{g+1}Q_0(x) $$
and
$$ P_1(x_1)=e_1^{-2}(c_1x_1+d_1)^{2g+2}\big(P_0(x)+(2m-4m^2)Q_0(x)^2\big)$$
where $x=(a_1x_1+b_1)/(c_1x_1+d_1)$. Then 
\[  y_1^2+Q_1(x_1)y_1=P_1(x_1) 
\]
is a minimal equation of $C$ over $A$. 
\end{enumerate}

\begin{remark} Let $y^2+Q(x)y=P(x)$ be the equation we start with. The
  change of variables to the minimal equation above is given as
  follows: 
$$ x= \dfrac{ax_1+b}{cx_1+d}, \quad y = \dfrac{ey_1+H(x_1)}{(cx_1+d)^{g+1}}$$ 
where
$$ \begin{pmatrix}  a & b \\ c & d \end{pmatrix} =M_0M_1, \quad e=e_0e_1,
$$
and $H(x_1)\in A[x_1]$ is determined by
$$   2H(x_1)=eQ_1(x_1)-(cx_1+d)^{g+1}Q((ax_1+b)/(cx_1+d)), 
$$
by comparing the traces of $y$ and of $y_1$ in the extension $K(x)[y]$
over $K(x)$.

The minimal discriminant $\Delta_{\mathrm{min}}$ is given by
$$\Delta_{\mathrm{min}}=e^{-4(2g+1)}(ad-bc)^{2(g+1)(2g+1)}\Delta$$ 
in terms of the initial discriminant $\Delta$ of Equation~\eqref{eq:start3}.
\end{remark}

\begin{example} Consider the equation
  $$ y^2+ 2^4\cdot 11 \cdot 13y=5^6\cdot 17^{3}x^5$$
  over $\mathbb Q$. It defines a genus $2$ curve of discriminant
  $$\Delta=2^{32}\cdot 5^{41}\cdot 11^8\cdot 13^8\cdot 17^{18}.$$
  The
command {\tt hyperellminimalmodel} in \cite{pari} gives 
a minimal equation over $\mathbb Z$:
$$
z^2+t^3z=1477440t^6 + 20t, 
$$
with the changes of variables: 
$$
x=\dfrac{2^2}{5\cdot 17t}, \quad 
y=\dfrac{2^4\cdot 5^3\cdot 17^2z+2^4\cdot 3^5\cdot 5^4\cdot 17^2t} {(5\cdot 17t)^3}
$$
with minimal discriminant equal to $2^{12}\cdot 5^{11}\cdot 11^8\cdot
13^8\cdot 17^8$. In particular, the initial equation is minimal away
from $2, 5$ and $17$. 
\end{example}

\begin{remark}
Our algorithm always terminates. The total  
number of iterations is roughly bounded above by
$\sum_{\p}v_\p(\Delta)$
where $\Delta$ is the discriminant of the initial equation. 
\end{remark}

\begin{remark} If in our algorithm we input an arbitrary pair of polynomials
  $Q(x), P(x) \in A[x]$, we must check if $y^2+Q(x)y=P(x)$ defines
  a smooth hyperelliptic curve $C$ and, if necessary,
      construct a new Weierstrass equation with polynomials $Q_0, P_0$ such that 
      $\deg Q_0\le g(C)+1$ and $\deg P_0\le 2g(C)+2$. 
  
Suppose for simplicity that  $\mathrm{char}(K)\ne 2$. Then the
smoothness is detected by the non-vanishing of
$\disc(4P+Q^2)$. Suppose from now on that this is the case. Then
$g=g(C)$ is given by $g=[(d-1)/2]$ where $d:=\deg (4P+Q^2)$.

Now if $\deg Q>g+1$, then $\deg P=2\deg Q>2g+2$. 
Write $Q(x)=Q_0(x)+2E(x)$ with $\deg Q_0\le g+1$ and $x^{g+2} \mid E$. 
Let $$y=y_0-E(x).$$
Then
$$
y_0^2+Q_0(x)y_0=P_0(x)
$$
where $P_0=P+QE-E^2$ satisfies $\deg P_0\le 2g+2$, is a new equation
of $C$ satisfying the requirement on the degrees of $Q_0, P_0$.

It reminds to show that $Q_o, P_0\in A[x]$. 
    It is enough to show that $E\in A[x]$. Consider a basis $\{ 1, z \}$ of the
    integral closure of $A[x]$ in $K(x,y)=K(C)$ with
    $\mathrm{Tr}_{K(C)/K(x)}(z) \in A[x]$ of degree $\le g+1$.
    See \cite{LTR}, Lemme (1.b), page 4579. Then $y=ez+H(x)$ with $e\in A$ and $H(x)\in A[x]$. So
    $$Q(x)=-\mathrm{Tr}_{K(C)/K(x)}(y)=-e\mathrm{Tr}_{K(C)/K(x)}(z)-2H(x)$$ 
    and $-E$ consists in the terms of degree $\ge g+2$ in $H(x)$.
    Therefore $E\in A[x]$. 
\end{remark}

\end{section}

\begin{section}{Minimization algorithm for pointed Weierstrass equations}

Fix $(C, w_0)$ as in \S~\ref{pointed-min}  and let
$$
y^2+Q(x)y=P(x) 
$$
be a pointed Weierstrass equation of $(C, w_0)$ over $A$ of discriminant
$\Delta$ (see the beginning of \S~\ref{mini-algo}). 
By Lemma~\ref{p-c}, if $v(\Delta)<4g(2g+1)$, then the equation is
pointed-minimal at $\p$. 

\begin{algo}[Even primes] \label{algo-p-even} \hskip 4pt

  {\bf Input:} The above equation.
  \smallskip

  {\bf Output:} A new pointed Weierstrass equation, minimal at even
  primes, and the change of variables $L_0(x)$ on $x$. 
  \smallskip
  
Let $\p_1, \dots, \p_n$ be the even prime divisors of $\Delta$.  Start
with $i=1$.  Let $L_0(x)=x$. 
  \begin{enumerate}[{\rm (I)}]
\item Let $\p=\p_i$. If $v(\Delta)<4g(g+1)$, go to (III).
\item If $\bar{Q}\ne 0$ or if $\bar{P}(x)$ is not a $(2g+1)$-th power,
 go to (III). Otherwise, $\bar{P}(x)=(x-\bar{c}_1)^{2g+1}$ for some
 $c_1\in A$. Run Algorithm~\ref{algo-lambda} at $\bar{c}_1$. 
  \begin{enumerate} 
  \item If $\lambda<2g+1$, go to (III).  
  \item Otherwise we have $\lambda=2g+1$. 
    Let
    $$ 
    Q_1(x)=\pi^{-g}Q(\pi x+c_1), \quad P_1(x)=\pi^{-2g}P(\pi x+c_1).$$
    \begin{enumerate}
  \item If $\overline{\pi^{-1}P}_1(x)$ is not a $(2g+1)$-th power, go to (III).
\item Otherwise let $\overline{\pi^{-1}P}_1(x)=(x-\bar{c})^{2g+1}$ for
  some $c\in A$. Run   Algorithm~\ref{algo-lambda} for the pair $Q_1,
  P_1$ at $\bar{c}$.  If $\lambda<2g+2$, go to (III). 
  \item Otherwise 
$$ \quad   Q(x)\gets \pi^{-(g+1)}Q_1(\pi x+c), \quad P(x)\gets
\pi^{-(2g+2)} P_1(\pi x+c),$$
$$ L_0(x) \gets \pi^2 L_0(x)+\pi c+c_1.$$ 
Go back to (II). 
\end{enumerate}
  \end{enumerate}
\item If $i<n$, then $i\gets i+1$, go back to (I).  Otherwise output $Q, P$
  and $L_0$. 
  \end{enumerate} 
\end{algo}

\begin{algo}[Odd primes] \label{algo-p-o} Suppose there are even and odd primes in $A$.

{\bf Input:} The pair $Q, P$ returned by Algorithm~\ref{algo-p-even}.
\smallskip

{\bf Output:} An equation minimal at all odd primes of $A$, together
with the changes of variables leading to the new equation. 
\smallskip   

Let
$$F=4P+Q^2.$$
(It has leading coefficient equal to $4$, but this does not matter.) 
Let $\p_1, \dots, \p_m$ be the odd prime divisors of $\Delta$. Let $L_1(x)=x$, let $i=1$. 
  \begin{enumerate}[{\rm (I)}] 
  \item Let $\p=\p_i$. If $v(\Delta)<4g(g+1)$, go to (III).
\item\label{algo-p-odd-2} If $\bar{F}$ has no root of order $2g+1$, go to  (\ref{algo-p-odd-R}).
Otherwise, let $\bar{c}\in k$ be the root of $\bar{F}$.
Write $$F(x)=\sum_i a_{c,i}(x-c)^i.$$  Let
$$\theta=\min \left\{ \dfrac{v(a_{c,i})}{(2g+1)-i}
  \ | \ 0\le i\le 2g \right\}.$$
Let $r=[\theta/2]$.
\begin{enumerate}
\item 
  If $r=0$, go to (\ref{algo-p-odd-R}).
\item Otherwise,   
  $$F(x) \gets \pi^{-2r(2g+1)}F(\pi^{2r}x+c), \quad
  y \gets \pi^{-r(2g+1)}$$
  $$ L_1(x)\gets \pi^{2r}L_1(x)+c$$ 
 and restart at (\ref{algo-p-odd-2}). 
\end{enumerate}
\item \label{algo-p-odd-R} If $i<m$, then $i\gets i+1$ and go back to (I). Otherwise
output $F(x)$ and $L_1(x)$. The equation $z^2=\frac{1}{4}F(x)$ is pointed-minimal
at all odd primes. 
\end{enumerate} 
\end{algo}

\medskip

\noindent{\bf Final step.} Suppose that $A$ has even and odd primes.
Denote by $Q_0, P_0$ be the pair $Q, P$ returned by
Algorithm~\ref{algo-p-even} and let
$L_1(x)=u_1^2x+c_1$ be returned by Algorithm~\ref{algo-p-o}.
Then $u_1$ is odd. If $u_1\in A^*$, then $y_0^2+Q_0(x)y_0=P_0(x)$ is pointed-minimal over $A$.
Suppose $u_1\notin A^*$. Let $m\in A$ be such that 
$$
4m \equiv 1 \mod u_1^{2g+1}.  
$$
Then if we define
$$Q_1(x_1):=u_1^{-(2g+1)}(1-4m)Q_0(u_1^2x_1+c_1)\in A[x_1], $$
$$P_1(x_1):=u_1^{-2(2g+1)}(P_0(u_1^2x_1+c_1)+
(2m-4m^2)Q_0(u_1^2x_1+b)^2)\in A[x_1]$$  
we have that
$$y_1^2+Q_1(x_1)y_1 = P_1(x_1)$$  
is a pointed-minimal equation of $(C, w_0)$ over $A$. 

\begin{remark} Let $y^2+Q(x)y=P(x)$ be the pointed equation we start
  with, of discriminant $\Delta$. The
changes of variables to the above minimal pointed equation is given as
follows: 
$$ x= u^2x_1+c, \quad y = u^{2g+1}y_1+H(x_1)$$ 
where
$$ 
u^2x_1+c=L_0(L_1(x_1)) 
$$
(choose any $u$ satisfying the above equation), and $H(x_1)\in A[x_1]$ is determined by
$$ 
  2H(x_1)=u^{2g+1}Q_1(x_1)-Q(u^2x_1+c), 
$$
by comparing the traces of $y$ and of $y_1$ in the extension $K(x)[y]$
over $K(x)$.
The minimal pointed discriminant $\Delta_{\mathrm{min}}$ is given    by
    $$
\Delta_{\mathrm{min}}=u^{-4g(2g+1)}\Delta.$$ 
\end{remark}

\end{section}

\begin{section}{Statements} 
\noindent No external dataset is used.
\medskip

\noindent No conflicts of interest to declare. 
\end{section}

\end{document}